\documentclass[reqno]{amsart}

\usepackage{hyperref}
\usepackage{amsmath,amssymb}
\usepackage{mathtools}
\usepackage[latin1]{inputenc}
\usepackage[T1]{fontenc}
\usepackage{graphicx}
\usepackage{soul}

\usepackage[all]{xy}
\usepackage{tikz}
\usepackage{pgfplots}

\DeclareMathOperator{\rank}{rank}

\DeclareMathOperator{\dist}{dist}

\DeclareMathOperator{\im}{Im}

\DeclareMathOperator{\chr}{Char}

\newcommand{\Ms}{\mathcal{M}}
\newcommand{\R}{\mathbb{R}}
\newcommand{\Ra}{\Rightarrow}

\newcommand{\Q}{\mathbb{Q}}
\newcommand{\N}{\mathbb{N}}
\renewcommand{\H}{\mathbb{H}}
\newcommand{\C}{\mathbb{C}}
\newcommand{\F}{\mathbb{F}}
\newcommand{\Z}{\mathbb{Z}}

\renewcommand{\epsilon}{\varepsilon}

\DeclarePairedDelimiter\absv{\lvert}{\rvert}
\DeclarePairedDelimiter\norm{\lVert}{\rVert}

\newtheorem{theorem}{Theorem}
\newtheorem{proposition}{Proposition}[section]
\newtheorem{lemma}[proposition]{Lemma}
\newtheorem{corollary}[proposition]{Corollary}

\theoremstyle{remark}
\newtheorem{remark}[proposition]{Remark}

\theoremstyle{definition}

\begin{document}

\title[Hpt. eq. of nearby Lagrangians and the Serre spectral sequence]{Homotopy equivalence of nearby Lagrangians and the Serre spectral sequence}
\author{Thomas Kragh}
\maketitle

\begin{abstract}
  We construct using relatively basic techniques a spectral sequence for exact Lagrangians in cotangent bundles similar to the one constructed by Fukaya, Seidel, and Smith. That spectral sequence was used to prove that exact relative spin Lagrangians in simply connected cotangent bundles with vanishing Maslov class are homology equivalent to the base (a similar result was also obtained by Nadler). The ideas in that paper were extended by Abouzaid who proved that vanishing Maslov class alone implies homotopy equivalence.

  In this paper we present a short proof of the fact that any exact Lagrangian with vanishing Maslov class is homology equivalent to the base and that the induced map on fundamental groups is an isomorphism. When the fundamental group of the base is pro-finite this implies homotopy equivalence.
\end{abstract}

\section{Introduction} \label{sec:introduction}

Let $L\subset T^*N$ be an exact Lagrangian embedding with $L$ and $N$ closed (compact without boundary). We will always assume that $N$ is connected, but for generality we will not assume that $L$ is connected. In \cite{FSS}, Fukaya, Seidel, and Smith constructed a spectral sequence converging to the Lagrangian intersection Floer homology of $L$ with itself, and used this to prove that exact relative spin Lagrangians in simply connected cotangent bundles with vanishing Maslov class are homology equivalent to the base (a similar result was simultaneously obtained by Nadler in \cite{Nadler}). This was extended by Abouzaid in \cite{Abou1} to prove that vanishing Maslov class implies homotopy equivalence (combined with the result in \cite{MySympfib} this actually proves homotopy equivalence for all exact Lagrangians). These approaches are rather technical and the goal of this paper is to prove a slightly weaker version in a much simpler way. To be precise we reprove the following theorem.

\begin{theorem}
  \label{thm:1}
  If $L\subset T^*N$ is a closed exact Lagrangian submanifold with vanishing Maslov class, then the map $L \to N$ is a homology equivalence and induces an isomorphism of fundamental groups.
\end{theorem}

\begin{remark}
  Note that, the theorem implies (by applying it to finite covers) that if the fundamental group of $N$ is pro-finite then $L\to N$ is a homotopy equivalence.
\end{remark}

We will prove the theorem by constructing a spectral sequence similar to the one used by Fukaya, Seidel, and Smith. We will construct this for any exact Lagrangian $L$ with any local coefficient system of vector spaces over some field $\F$ (and with a relative pin structure when needed).

The construction of this spectral sequence goes as follows. We start with a Morse function (with some restrictions that we will not write out here) $g:N \to \R$ and consider two large scale perturbations of $L$ given by
\begin{align*}
  K_t &= t L \\
  L_t &= t L + dg
\end{align*}
for very small $t>0$. So $K_t$ is a scaling of $L$ by a very small constant making it very close to the zero section, and $L_t$ is the same but pushed off the zero-section using the Morse function $g$, so that it is close to the graph of $dg$ instead. This is illustrated in Figure~\ref{fig:1a} close to a critical point $q$ of $g$.
\begin{figure}[ht]
  \centering
  \begin{tikzpicture}
    \draw[->] (-2,0) -- (2,0) node [right] {$N$};
    \draw[->] (0,-1.4,0) -- (0,1.7) node [below left] {$T^*_{q}N$};
    \fill (0,0) circle (1pt) node [below right] {$q$};
    \draw plot [smooth] coordinates {(-2,1.1) (0.8,0.7) (-0.8,-0.5) (2,-0.9)};
    \draw (0.8,0.7) node [above right] {$L$};
  \end{tikzpicture}
  \begin{tikzpicture}
    \draw[->] (-2,0) -- (2,0) node [right] {$N$};
    \draw[->] (0,-1.4,0) -- (0,1.7) node [below left] {$T^*_{q}N$};
    \draw (0,0) node [below right] {$q$};
    \draw plot [smooth] coordinates {(-1.5,1.59) (0.8,-0.73) (-0.8,0.75) (1.5,-1.57)};
    \draw plot [smooth] coordinates {(-2,0.11) (0.8,0.07) (-0.8,-0.05) (2,-0.09)};
    \draw (-0.8,0.7) node [above right] {$L_t$};
    \draw (-0.8,0) node [below] {$K_t$};
  \end{tikzpicture}
  \caption{Intersections of $K_t$ and $L_t$ close to a critical point $q$ of $g$.}
  \label{fig:1a}
\end{figure}
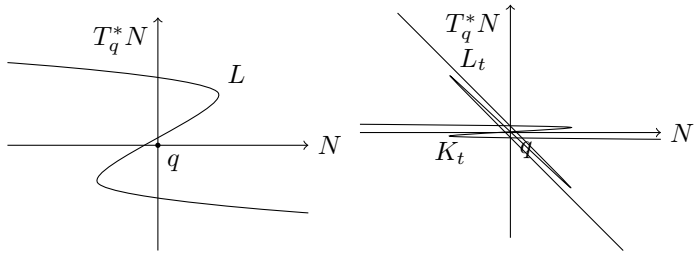
As the figure illustrates all the intersection points of the two Lagrangians will ``bunch'' around the critical points of $g$. Each of the intersections points in the bunch close to $q$ will have action close to the critical value $g(q)$ (up to an overall shift that we thus fix). For small $t$ we now consider an action filtration such that we have a non-trivial filtration level for each critical \emph{value} of $g$ and it contains all the intersection points in all bunches with action value close to this critical value. So, each filtration level contains an unspecified number of these bunches.

The main technical part of the construction is carried out in Section~\ref{sec:fiber-wise-self}. There we basically prove that each of the bunches on the same filtration level do not interact (with respect to the differential), and that restricting the differential to any bunch is well-defined and that this always produces the same homology groups - up to a shift by the Morse index of the associated critical point of $g$. In fact, we will use this ``bunching'' construction to create a local system on $N$. This we will use in Section~\ref{sec:spectral-sequence} to prove that, for $g$ self-indexing, page two of the associated spectral sequence looks a lot like a Serre spectral sequence. In fact we will identify page 1 as the Morse homology complex of $g$ with coefficients in the local system defined in Section~\ref{sec:fiber-wise-self}. 

The original spectral sequence by Fukaya, Seidel, and Smith did not look as much as a Serre spectral as the one in this paper - so we now explain the difference. Consider the following two filtrations defined for a fibration $\pi : E\to X$ where $X$ is a finite cell complex by:
\begin{itemize}
\item $\pi^{-1}(X_0) \subset \cdots \subset \pi^{-1}(X_j) \subset \pi^{-1}(X_{j+1}) \subset \cdots \subset E$, where $X_j$ denotes the $j$-skeleton. This defines the Serre spectral sequence (see Hatcher~\cite{HatchSpec}).
\item Similar, except $X_j$ is not the $j$-skeleton, but $X_{j+1}$ is $X_j$ with a single new cell. This leads to a spectral sequence analogous to the one by Fukaya, Seidel, and Smith with a filtration level per cell - not necessarily ordered by dimension.
\end{itemize}
In a Morse theoretic construction this corresponds to the two cases:
\begin{itemize}
\item Self-indexing Morse function - where the critical value equals the Morse index.
\item Any Morse function with distinct critical value for each critical point.
\end{itemize}
The relation between these two viewpoints (and its analogy to the bunching of critical points) can be described as follows. Assume $E\to X$ is a fiber-bundle of closed manifolds. Let $f$ be a Morse function on $X$. This makes $f'=(f\circ \pi) : E \to \R$ a Morse-Bott function. Perturbing $f'$ slightly to make it Morse we get bunches of critical points each close to one of the original critical fibers, and we may define a filtration on the Morse complex of $f'$ by using a sequence of values intertwining the original critical values of $f$ in such a way that all the bunches associated to the same original critical value are in the same filtration level. By standard perturbation arguments the individual bunches close to fibers over critical points with the same critical value do not interact in the differential, and the local system produced by a construction similar to that in Section~\ref{sec:fiber-wise-self} will simply be the homologies of the fibers.

The final piece to proving Theorem~\ref{thm:1} is essentially to establish a version of Poincar\'e duality fiber-wise. A heuristic description of why this might be a useful property is as follows; the fibers represent the relative difference $L \to N$. However if the fibers also behave as a manifold - then this is homologically supposed to look like a fiber-bundle, and since $L$ and $N$ have the same dimension the fiber basically (homologically) has to be a $0$ dimensional manifold. A similar argument was used in the simply connected case by Fukaya, Seidel, and Smith.

The general layout of the paper is as follows. In section~\ref{sec:inters-floer-homol} we briefly describe Lagrangian intersection Floer homology of two exact Lagrangians $K$ and $L$ and how to apply local coefficients. In Section~\ref{sec:fiber-wise-self} we define the fiber-wise (over $N$) intersection Floer homology of any Lagrangian $L$ with itself using the idea of the bunches described above. We also prove that this fiber-wise Floer homology defines a graded local system on $N$, and satisfies other natural properties that we will need - most importantly the Poincaré duality mentioned above. In Section~\ref{sec:spectral-sequence} we use action filtrations to construct the spectral sequence as described above, converging to the full Lagrangian intersection Floer homology; and we also identify page 1 of this spectral sequence for special cases of $g$. In Section~\ref{sec:local-syst-univ} we extend the type of local coefficient systems we allow to include local systems on the universal cover of $N$. The reader only interested in the case where $N$ is simply connected can skip this section. Then in Section~\ref{sec:consequences} we prove Theorem~\ref{thm:1} starting with the simply connected case not requiring Section~\ref{sec:local-syst-univ}.

It should be noted that the ideas used in this construction are similar to the original ideas behind the spectral sequence constructed by Fukaya, Seidel, and Smith.

\begin{remark}
  In the paper \cite{sequel} with Abouzaid we use the same large scale perturbations of $L$ above together with some additional structure to prove the new result that any exact Lagrangian is in fact simple homotopy equivalent to the base.
\end{remark}


\subsection*{Acknowledgments}
  I would like to thank Tobias Ekholm for many insightful discussions on the topic. I would also like to thank the anonymous referee and Maksim Maydanskiy for suggestions which led to a much better exposition of the material.
\section{Lagrangian Intersection Floer homology and local coefficients} \label{sec:inters-floer-homol}

In \cite{MR965228}, Floer introduced the Lagrangian intersection Floer homology $HF_*(K,L)$; and proved that it is a Hamiltonian isotopy invariant. In this section we recall this construction for two exact Lagrangians $K,L \subset T^*N$. This also serves to fix some conventions regarding signs, gradings and orientations. We will consider some ground field $\F$. However, we will consider any local coefficient systems of vector spaces over $\F$ defined on $K$ or $L$, and describe (Corollary~\ref{cor:Grading:1}) the generalization of Floer's result that
\begin{align} \label{eq:3}
  HF_*(L,L;\F) \cong H_*(L;\F).
\end{align}
to such local coefficient system. Formally we consider the local systems on $K$ and $L$ to be non-graded or, equivalently, graded in degree 0.

The reader only interested in the case of both $L$ and $N$ simply connected can ignore the local coefficients in this section. However, we note that we still need to specifically identify a certain differential in the spectral sequence in Proposition~\ref{prop:Spectral:1}, which means we need to understand the fiber-wise Floer homology defined in the next section as a graded local system on the base $N$. So, one cannot avoid local coefficients in this argument, and hence it does not simplify matters much to ignore them here.

Let $N$ be any closed (compact without boundary) manifold. The canonical 1-form (or Liouville form) $\lambda \in \Omega^1(T^*N)$ on the cotangent bundle is defined by
\begin{align*}
  \lambda_{q,p}(v) = p(\pi_*(v)), \qquad q\in N,\, p\in T^*_qN,\, \pi :T^*N \to N.
\end{align*}
The canonical symplectic form is then given by $\omega = -d\lambda$. Pick a Riemannian structure on $N$, then we get an induced almost complex structure $J$ on $T^*N$ (which is compatible with the symplectic structure). This canonically identifies $T_{(q,p)}T^*N \cong \C \otimes T_qN$, where the real part is horizontal and the imaginary part is vertical.

For such a $J$ there is a canonical map from the space of linear Lagrangians subspaces $V \subset T_{q,p}(T^*N)$ to $S^1$ given by the square determinant. Indeed, pick any orthonormal basis for $T_qN$ then this represents a basis of the horizontal Lagrangian at $T_{q,p}(T^*N)$. Now also pick an orthonormal basis for $V$. The complex unitary linear map changing from the first basis to the second describes a unique element in $U(n)$ of which we can take the square determinant. This is independent on the choice of both bases since it is invariant under both actions by $O(n)$.

This is smooth in $V$ and $(q,p)$, and thus it induces a smooth map from any Lagrangian submanifold $L\subset T^*N$ to $S^1$, by sending $z\in L$ to the number defined by $T_zL$. The induced map on $\pi_1$ or $H_1$ is known as the Maslov class. For a Lagrangian submanifold $L\subset T^*N$ with vanishing Maslov class a \emph{grading} $\psi$ (defined in \cite{MR1765826}) is a lift $\psi : L\to \R$ of the map $L\to S^1 \cong \R/\Z$ defined above. From now on we assume that $K$ and $L$ are two exact Lagrangians with vanishing Maslov classes and gradings $\psi_K$ and $\psi_L$. Notice that when one has an isotopy of Lagrangians $L_t,t\in I$ then a grading on $L_0$ ``parallel transports'' to a unique grading on each $L_1$.

\begin{remark}
  \label{rem:Grading:1}
  Everything in this paper except Section~\ref{sec:consequences} can be carried out in the general case (with modified grading), but we assume vanishing Maslov classes already here to make the exposition more clear.
\end{remark}

Now let $z\in K\cap L$ be a transverse intersection point of $K$ and $L$. Since the space of linear Lagrangians in $T_zT^*N$ which are transverse to $T_zK$ is contractible there is a path unique up to homotopy in this space from $T_zL$ to $J(T_zK)$. This path lifts using the determinant construction above to a path from $\psi_L(z)$ to some other real number $a$. We now define the grading of $z$ (dependent on the order $K$ before $L$) by the formula
\begin{align*}
  \deg_{(K,L)}(z) = \psi_K(z)-a+\tfrac{n}{2}.
\end{align*}
This is an integer since $\det(JA)^2=(-1)^n\det(A)^2$, and each $(-1)$ represents a half turn around $S^1$. Notice that with this convention it is an easy exercise to see that pushing the zero-section $N$ off itself using a Morse function $g$ makes the intersection points between $N$ and $dg$ (in that order) have grading equal to the Morse index of $g$ (here the grading on $dg$ is induced from $N$ by the obvious isotopy). This grading also satisfies:
\begin{align} \label{eq:14}
  \deg_{(K,L)}(z) = n-\deg_{(L,K)}(z).
\end{align}
However, when the order is implied from context we will simply write $\deg(z)$.

Now assume that we have two transverse intersection points $z_{-1},z_1\in K\cap L$. Let $\H$ be the upper half plane in $\C$ and consider the space $\chi(z_{-1},z_1)$ of maps $u:D^2\to T^*N$ such that $u$
\begin{itemize}
\item maps $\pm 1$ to $z_{\pm 1}$,
\item maps the lower edge to $K$ -- i.e. $u(S^1\cap \overline{\H}) \subset K$, and
\item maps the upper edge to $L$ -- i.e. $u(S^1\cap \H) \subset L$.
\end{itemize}
and let $\Ms(z_{-1},z_1)\subset \chi(z_{-1},z_1)$ be the subspace of pseudo-holomorphic maps. For generic $J$ (usually achieved by a small perturbation) this subspace is a manifold of dimension $\deg(z_1)-\deg(z_{-1})$. Indeed, this is the Fredholm index of the linearization of the $\overline{\partial}$ operator.

The space $\Ms(z_{-1},z_1)$ has an $\R$ action (symmetries of holomorphic disc with two marked points on the boundary), which when the map is non-constant is free. So in the case where $\deg(z_1)-\deg(z_{-1})=1$ the quotient $\Ms(z_{-1},z_1)/\R$ is for generic $J$ a manifold of dimension $0$ - we refer to these points as rigid discs. For $K$ transverse to $L$ and generic $J$ (which we assume for the rest of this section) Floer defined the chain complex:
\begin{align*}
  CF_*(K,L;\F_2) = (\F_2[K\cap L],\partial)
\end{align*}
Here 
\begin{itemize}
\item $\F_2=\Z/2$, but we will describe more general coefficients later,
\item the grading is given by $\deg_{(K,L)}$, and
\item $\partial$ counts the number of rigid discs between the intersection points going down in degree - I.e. $\deg(z_{1})=\deg(z_{-1})+1$.
\end{itemize}
By the assumptions the space $\Ms(z_{-1},z_{1})/\R$ when $\deg(z_{1})-\deg(z_{-1})=2$ is a 1-manifold. A version of Gromov compactness and gluing of discs shows that it can be compactified to a manifold with boundary by adding the boundary:
\begin{align*}
  \smashoperator{\bigsqcup_{z\in K\cap L,\deg(z)=\deg(z_{1})-1}} \Ms(z_{1},z)\times\Ms(z,z_{-1}),
\end{align*}
which is a complete analogue of the Morse homology complex situation when gluing gradient trajectories. Indeed, the boundary structure is given by gluing discs together in a similar fashion. This is thus used to prove that $\partial^2=0$ - as in Morse homology. Floer also proved that this homology is invariant under Hamiltonian isotopy of either $K$ or $L$.

Now consider any local coefficient system $C$ of $\F_2$ vector spaces on $K$ (or $L$), and define the chain complex 
\begin{align*}
  CF_*(K,L;C) = (\smashoperator[r]{\bigoplus_{z\in K\cap L}} C_{z},\partial).
\end{align*}
Here the differential is again defined by counting the rigid discs, but using the parallel transport in the local system along the boundary path of the disc in $K$ (or $L$). The proof that $\partial^2=0$ easily extends to this case since the boundary path in $K$ of any glued disc is up to homotopy given by the concatenations. It should be noted that this works even when the local system is infinite dimensional, which we will make use of in Section~\ref{sec:local-syst-univ} and Section~\ref{sec:consequences}. This was first observed by Damian in \cite{MR2914855} and Abouzaid in \cite{Abou1}.

When $K$ and $L$ are not transverse one defines this by perturbing one of them by a Hamiltonian flow. Floer proved that if $K=L$ with the same grading then
\begin{align} \label{eq:2}
  HF_*(L,L;\F_2) \simeq H_*(L;\F_2).
\end{align}
In fact, by using a $C^2$ small Morse function $g : L \to \R$ to push $L$ off itself  (and changing $J$) Floer proved that
\begin{align*}
  CF_*(L,L;\F_2) \cong CM_*(L;\F_2).
\end{align*}
Here $CM_*$ denotes Morse complex of $g$. Floer proved this by proving that the gradient trajectories of $g$ are in bijective correspondence with the (now very narrow) pseudo-holomorphic discs with both boundaries equal to the gradient trajectory, and as we saw above the degree matches the Morse index. Since this also explicitly describes the boundaries of the discs involved in the differential we conclude that for a local system this proof extends to proving that
\begin{align*}
  CF_*(L,L;C) \cong CM_*(L;C),
\end{align*}
when $C$ is a local system on either of the two copies of $L$, implying that
\begin{align*}
  HF_*(L,L;C) \cong H_*(L;C).
\end{align*}

To define the intersection Floer homology with coefficients not $2$ torsion (local or not) one needs to count the rigid discs with signs, and to do this one needs that the Fredholm index bundle (which leads to the above discussed Fredholm index) has a trivialization of its determinant line bundle (as discussed in \cite{MR1200162}) which is compatible with gluing. To achieve this we need to choose relative pin structures on $K$ and $L$ (see e.g. \cite{MR2441780} or \cite{MR2553465}). We will use the conventions from \cite{MR2441780} and define
\begin{align*}
  CF_*(L,L;C) = (\smashoperator[r]{\bigoplus_{z\in K\cap L}} \absv{o_z} \otimes C_{z},\partial),
\end{align*}
where $\absv{o_z}$ denotes the infinite cyclic group generated by the orientations (representing generators with opposite signs) of a certain line $o_z$ as defined in \cite{MR2441780} (sections 12b and 12f). We note that in the case of $N$ and $dg$ we can canonically identify the two possible generators of $\absv{o_z}$ with orientations on the negative eigen-space of the Hessian of $g$. This is a key ingredient in defining the signs in Morse homology away from characteristic 2. The pin structures allow us to associate a canonical isomorphism
\begin{align*}
  o_{z_{-1}} \cong o_{z_1} \qquad \Ra \qquad  \absv{o_{z_{-1}}} \cong \absv{o_{z_1}}
\end{align*}
to each rigid disc as above. To define the differential $\partial$ we now sum the latter maps tensored with the induced maps on the local system $C$ from before. The relative pin structures also allow us to associate a compatible orientation on the 1-manifolds in the proof of $\partial^2=0$, which means that that proof extends to this case. Even the Hamiltonian invariance generalizes.

\begin{remark} \label{rem:Grading:2}
  Note that the sign conventions in \cite{MR2441780} are such that if one changes the grading of a Lagrangian by adding 1 to the lift then all the signs on the differentials change, which means that by $C[1]$ we will mean the shift of the chain complex $C$ with the negative differential.
\end{remark}

Floer's proof extends to signs (given the same relative pin structure on both copies of $L$) in the sense that the signs equal the signs in the Morse complex. So, his proof immediately generalizes to show the following corollary. 

\begin{corollary}
  \label{cor:Grading:1}
  Let $C$ be a local system on $L$. If $C$ is 2-torsion \emph{or} $L$ is equipped with a relative pin structure then 
  \begin{align}
    \label{eq:4}
    HF_*(L,L;C) \cong H_*(L;C).
  \end{align}
\end{corollary}


\section{Fiber-wise intersection Floer homology} \label{sec:fiber-wise-self}

\newcommand{\oL}{{\widetilde{L}}}

Let $q\in N$ be any point, and let $V^m\subset T_qN$ be an $m$-dimensional linear subspace. In this section we define the fiber-wise self-intersection Floer homology
\begin{align*}
  HF_*(L,q,V^m;C)
\end{align*}
of a graded exact Lagrangian $L$ (with relative pin structure when necessary). Here $C$ is a field $\F$ or more generally a local coefficient system of $\F$-vector spaces on $L$ (potentially infinite dimensional). Initially this fiber-wise Floer homology will depend on a lot of other choices, of which the most important is a function $g$ with $q$ as a Morse critical point with unstable manifold tangent to $V^m$. We then prove that these fiber-wise intersection Floer homology groups are independent of the auxiliary choices and satisfy the following properties, which we will need in the proof of Theorem~\ref{thm:1}.
\begin{itemize}
\item Invariance: $HF_*(L,\bullet,-;C)$ canonically defines a graded local system on the Grassmann bundle of choices $(q,V^m)$.
\item Morse shifting: $HF_*(L,q,V^m;C) \cong HF_{*+m}(L,q,0;C)$ (sign dependent on a choice of an orientation of $V^m$).
\item Poincare duality: $HF_*(L,q,V^m;C^\dagger) \cong HF_{n-*}(L,q,(V^m)^\perp;C)^\dagger$.
\end{itemize}
Here the latter $(-)^\dagger$ is vector space dual. However, $C^\dagger$ denotes the dual local system, which is defined by taking the fiber-wise dual and tensoring with the rank 1 local system associated to local orientations of $L$. This latter local system is trivial if $L$ is orientable with respect to $\F$.
\begin{remark}
  It is a consequence of vanishing Maslov class that the contribution of the orientation line (or dualizing sheaf) of $L$ over a point $q\in N$ is trivial. This implies that for this version of Poincare duality we actually do not need to tensor with this orientation line. However, to avoid a lengthy sign discussion we simply state it as above and refer to \cite{MR2441780} for the signs.
\end{remark}

In Section~\ref{sec:inters-floer-homol} we fixed a Riemannian structure on $N$ inducing an almost complex structure $J$ on $T^*N$. Let
\begin{align} \label{eq:8}
  g : N \to \R
\end{align}
be a smooth function which has $q$ as a non-degenerate critical point, and whose Hessian has negative eigenspace equal to $V^m$. This is easily constructed using a normal neighborhood of $q$, and it is a contractible choice.

Define
\begin{align*}
  K_t &= t L \\
  L_t &= t L + dg,
\end{align*}
which we will consider for very small $t>0$. Here $dg : N \to T^*N$ is a Lagrangian, but $+dg$ means that we shift a point $v\in T^*N$ to $v +dg(\pi(v))$. This is the same as the Hamiltonian time 1 flow using the Hamiltonian $ (g\circ \pi) : T^*N \to \R$. So, $K_t$ and $L_t$ are both Hamiltonian isotopic to $L$. Fix a primitive $f^L : L \to \R$ for the restrictions of $\lambda$, then
\begin{align*}
  f^{K_t}(z) &= t f^L(t^{-1} z) \\
  f^{L_t}(z) &= t f^L(t^{-1} (z-dg_{\pi(z)})) + g(\pi(z)) 
\end{align*}
will be used as primitives for $\lambda$ on $K_t$ and $L_t$ respectively.

Using the canonical identification we can transport $C$ and the gradings to corresponding structures on $K_t$ and $L_t$. The intersection Floer homology with these structures can be defined as in Section~\ref{sec:inters-floer-homol}. However, for small $t$ we get that the intersections of $K_t$ and $L_t$ are close to critical points of $g$. Indeed, $K_t$ is close to the zero-section and $L_t$ is close to $dg$ so only when $dg$ is close to $0$ do they intersect (see Figure~\ref{fig:1a}). For small $t>0$ we will call the intersection points close to $q$ the \emph{bunch} of intersection points associated to $q$.

The action of an intersection point $z\in K_t\cap L_t$ is given by the difference of the primitives:
\begin{align} \label{eq:10}
  f^{K_t}(z)-f^{L_t}(z) = t(f^L(z_+)-f^L(z_-)) + g(\pi(z)),
\end{align}
where $z_- = t^{-1}z \in K$ and $z_+ \in L$ is the solution to $tz_++dg = z$. Since $f^L$ is bounded this means that the critical action values will for small $t$ cluster around the critical values of $g$. More importantly, the action values of the bunch associated to $q$ cluster around $g(q)$. This means that the action interval of the bunch is very narrow, and the following lemma will be used to argue that restricting the Floer chain complex to only include the intersection points in this bunch gives a well-defined complex for small $t$. However, we formulate it using any function $f$ with any isolated singularity at $q$ since we will need this later.

\begin{lemma}
  \label{lem:Fiberwise:1}
  Let $f:N \to \R$ be any function such that $q$ is the only critical point in the closure of the ball $B_R(q)$. Then there exist a $\delta>0$ and an $a>0$ such that: if $u:D^2\to T^*N$ is a pseudo-holomorphic disc satisfying:
  \begin{itemize}
  \item Precisely one of the two points $u(\pm 1)$ is in the cotangent ball $T^*B_R(q)$.
  \item The maximal distance of the upper boundary of $u$ to $df$ is $\delta$.
  \item The maximal distance of the lower boundary of $u$ to the zero-section is $\delta$.
  \end{itemize}
  then the symplectic area of $u$ is larger than $a$. Furthermore, neither of the points $u(\pm 1)$ is in the set
  $$T^*(B_R(q)-B_{R/2}(q)).$$
\end{lemma}

\begin{proof}
  The assumptions imply that for small $\delta>0$ the one point of $u(\pm 1)$ that lies inside $T^*B_R(q)$ is in fact inside $T^*B_{R/2}(q)$. Indeed, there is a positive distance from the closed annulus $\overline{B_R(q)-B_{R/2}(q)} \subset N \subset T^*N$ to $df$. So we may choose $\delta$ to be smaller than half this distance.

  Consider the co-dimension 2 sub-manifold $W\subset T^*B_R(q)$ given by 
  \begin{align} \label{eq:1}
    W = \{(q',p')\in T^*N \mid \dist(q',q) = \epsilon, \norm{p'}=\tfrac12 \norm{d_{q'}f} \}.
  \end{align}
  for $R/2 <\epsilon <R$. The manifold $W$ is compact without boundary, and it is disjoint from $df$ and the zero section. By the above choice of $\delta$ it is also disjoint from $u(S^1)$ when $u$ is as described in the lemma. In fact, the upper part of the boundary can only pass over points of $W$ (here ``over'' means with larger $\norm{p}$ value) and the lower part of the boundary passes under. Hence the assumptions imply that the $\Z/2$ algebraic intersection of $u$ and $W$ is $1$.

  It follows by standard monotonicity (Lemma~\ref{lem:Fiberwise:6}) that $u$ has area at least $a$ for some small $a>0$.
\end{proof}

\begin{lemma}
  \label{lem:Fiberwise:6}
  Let $M$ be any open symplectic manifold with a compatible almost complex structure $J$. Then for any compact subset $C\subset M$ and an open neighborhood $U$ around $C$ there is a lower bound on the area of any non-constant connected pseudo-holomorphic curve passing through $C$ defined on an open domain and with proper image in $U$.
\end{lemma}

\begin{proof}
  This was proven (but not phrased like this) in \cite{MR809718}.
\end{proof}

Now let $a>0$ and $\delta>0$ be as in Lemma~\ref{lem:Fiberwise:1} for $g$ with some $R>0$ isolating $q$ from other critical points. We now use these to define the fiber-wise intersection Floer homology. Firstly, pick $t$ so small that any disc $u$ with upper boundary on $L_t$ and lower boundary on $K_t$ satisfies the $\delta$ distance bound in the lemma, but also such that the critical actions of the bunch associated to $q$ lies in the interval $[g(q)-a/4,g(q)+a/4]$. Then define the fiber-wise intersection Floer homology:
\begin{align*}
  CF_*(L,q,V^m;C) = &CF_*(L,q,V^m,g,t,H,J';C) = (\smashoperator{\bigoplus_{z\in K_t\cap L_t'\cap T^*B_{R/2}(q)}} \absv{o_{z}}\otimes C_{z},\partial_{\mid}).
\end{align*}
Here $L_t'$ is a small Hamiltonian perturbation of $L_t$ using a $C^2$ small Hamiltonian $H$, and $J'$ is a generic small perturbation of $J$. The differential $\partial_{\mid}$ is the restriction of the differential discussed in Section~\ref{sec:inters-floer-homol}.

\begin{lemma}\label{lem:Fiberwise:7}
  The fiber-wise intersection Floer homology is well-defined and independent of the choices up to a chain homotopy equivalence, which is unique up to chain homotopy.
\end{lemma}

\begin{proof}
  Initially we consider $g$ as fixed. For small enough Hamiltonian perturbation we can assume that all the intersection points of $K_t$ and $L_t'$ in the bunch have action in $(g(q)-a/3,g(q)+a/3)$ and that $L_t'$ also lies $\delta$ close to $dg$. Lemma~\ref{lem:Fiberwise:1} was used in the definition above for the original $J$. However, for small perturbations of $J$ we can assume that any pseudo-holomorphic disc with precisely one of the points $u(\pm 1)$ in the bunch and boundaries on $K_t$ and $L_t'$ has symplectic area larger than $2a/3$, which is still more than the entire interval of critical action spanned by these critical points - hence there are no interactions from outside the bunch. More concisely, the usual proof that $\partial_{\mid}^2 = 0$ works unchanged since there can be no breaking on this subset of generators which involves points outside of the bunch.

  It is standard to construct continuation maps for intersection Floer homology using generic paths of perturbation data (see e.g. \cite{MR2441780}). If all the perturbations in the path are small enough the bound in Lemma~\ref{lem:Fiberwise:1} is valid also for the associated continuation map. Hence this map restricts to a chain map on the fiber-wise Floer complexes. Furthermore, since generic homotopies of such paths induce chain homotopies of these continuation maps it follows that for small enough perturbations the continuation maps are chain homotopy equivalences which are unique up to homotopy.

  Now, we consider the choice of $g$ and note that this is a contractible choice, so for any two choices there is a path $g_s,s\in I$ between them, and since changing $g$ slightly changes $K_t$ and $L_t$ by a slight perturbation we can cut $I$ into small pieces and get a sequence of chain homotopy equivalences (each as above for small $t$) relating the two chain complexes. Since the path $g_s$ is unique up to homotopy, we can relate any such two choices by a homotopy of paths, which when cut into pieces can be used to define a chain homotopy between the two sequences of chain homotopy equivalences.
\end{proof}

Let $E^m \to N$ be the Grassmann bundle with fibers $E^m_q$ the $m$ dimensional linear sub-spaces of $T_qN$. Hence $E^0=E^n=N$.

\begin{lemma}
  \label{lem:Fiberwise:2}
  The fiber-wise intersection Floer homology
  \begin{align*}
    HF_*(L,q,V^m;C) = H_*(CF_*(L,q,V^m,g,t,H,J';C))
  \end{align*}
  naturally defines a graded local system on the choices $(q,V^m) \in E^m$.
\end{lemma}

\begin{proof}
  Let $g_s :N \to \R$ be a smooth family of functions parametrised by $s=(q_s,V^m_s) \in E^m$ such that $g_s(q_s)$ is a Morse critical point with the negative eigenspace of the Hessian equal to $V^m_s$. This can be constructed explicitly using exponential maps and bump functions. By compactness of $E^m$ we can find an $R>0$ such that for each $s$ the critical point of $g_s$ at $q_s$ is unique in the closure of $B_R(q_s)$. By compactness we can find a $\delta>0$ and an $a>0$ as in Lemma~\ref{lem:Fiberwise:1} which works for the entire family $g_s,s\in E^m$ simultaneously, and again we can find $t$ so small that \emph{all} the fiber-wise Floer homologies are well-defined (each after a perturbation). Since changing $s$ slightly perturbs $K_t$ and $L_t$ slightly it follows from Lemma~\ref{lem:Fiberwise:7} that the \emph{homologies} of the complexes 
  \begin{align*}
    CF_*(L,q_s,V_s^m,g_s,t,H,J';C)
  \end{align*}
  are locally defined up to unique isomorphism for $s\in E^m$, and hence defines a local system on $E^m$.
\end{proof}

\begin{lemma}\label{lem:Fiberwise:8}
  Let $v \in  V^m$ be a unit vector let $V^{m-1}\subset V^m$ be the orthogonal complement of $v$. There is a chain homotopy equivalence (after choosing small perturbations)
  \begin{align*}
    CF_*(L,q,V^m;C) \simeq CF_{*-1}(L,q,V^{m-1};C)
  \end{align*}
  unique up to chain homotopy. Moreover, the induced map on homology is continuous for varying $v, V^m$ and $q$ (and hence $V^{m-1}$).
\end{lemma}

\begin{remark}
  Notice here that the continuity and uniqueness of the map on homology is equivalent to: on the space of choices $(q,V^m,v)$ we have two fiber-bundle structures given by projections to $E^m$ (with fiber $S^{m-1}$ - the choice of $v$) and to $E^{m-1}$ (with fiber $S^{n-m}$ - since $v$ is a choice of unit vector in the orthogonal complement of $V^{m-1}$). Now, the construction defines a canonical global isomorphism of local systems between the pull backs of the two local systems of fiber-wise Floer homologies on $E^m$ and $E^{m-1}$ to the common fiber bundle.
\end{remark}

\begin{proof}
  Fix $q\in N$. We will work in a normal coordinate chart around $q$ which identify the derivatives $\tfrac{\partial}{\partial x_i},i=1,\dots,m$ with $V^m$ and such that $\tfrac{\partial}{x_m}$ is mapped to $v$. We will denote the image of the span of the first $m-1$ of these by $V^{m-1}$, which is the orthogonal complement to the span of $v$ inside $V^m$. Let $g_s$ be a family of functions for $s\in (-\epsilon,\epsilon)$ which in these coordinates is given by
  \begin{align*}
    g_s(x) = -x_1^2-\cdots - x_{m-1}^2 +(x_m^2-s)x_m+x_{m+1}^2+\cdots x_n^2.
  \end{align*}
  This has two critical points in the chart when $s>0$ and none when $s<0$. Defining the Lagrangians as above using $g_s$ instead of $g$ and some small $t>0$ provides smooth families $L_t^s$ and $K_t^s$ of Lagrangians.

  For any $\delta>0$ there is an $\epsilon>0$ small so that for $0<t<\epsilon$ and $s\in [-\epsilon,\epsilon]$ all of the Lagrangians $L_t^s$ are within a $\delta$-neighborhood of $dg_0$ and $K_t^s$ within a $\delta$-neighborhood of the zero-section. Hence using Lemma~\ref{lem:Fiberwise:1} on $g_0$ (and some $R>0$) provides an $a>0$ (for our fixed $J$) which we can use for this family of Lagrangian pairs. By making $\epsilon$ even smaller we get that the critical action interval of intersection points of $L_t^s\cap K_t^s$ in $T^*B_R(q)$ is again smaller than $2a/3$, and thus for any such pair $(s,t)$ the Floer homology complex, say $SCF_*$ (``S'' for singularity), of all the intersection points inside $T^*B_R(q)$ is well-defined (using a sufficiently small perturbation). So, as above, this ``singularity Floer homology'', say $SHF_*$, defines a graded local system on the space $(s,t)\in [-\epsilon,\epsilon]\times]0,\epsilon]$.

  For $s=-\epsilon$ and $t$ sufficiently small we see that $L_t^s\cap K_t^s \cap T^*B_R(q) = \varnothing$ and so $SHF_*$ must be the trivial local system.

  For $s=\epsilon$ the function $g_s$ has two Morse-critical points $z_1$ and $z_2$ close to $q$ with $g_s(z_1)<g_s(z_2)$. Since the critical action values will cluster around these values we see that for small $t$ we have a block form differential on $SCF_*$:
  \begin{align*}
    d =
    \begin{pmatrix}
      d^m & F \\
      0 & d^{m-1} 
    \end{pmatrix}
  \end{align*}
  Here $d^m$ and $d^{m-1}$ are the differentials in the fiber-wise Floer homology chain complexes $CF_*(K,L,z_1,V^m;C)$ and $CF_*(K,L,z_2,V^{m-1};C)$ at the points $z_1$ and $z_2$ respectively. By the sign convention discussed in Remark~\ref{rem:Grading:2} it follows that 
  \begin{align} \label{eq:13}
    F_v^m=F : CF_*(L,z_1,V^m;C)[-1] \to CF_*(L,z_2,V^{m-1};C)    
  \end{align}
  is a chain homotopy equivalence. Both points are very close to $q$, so we may replace $z_1$ and $z_2$ with $q$ using continuation maps unique up to chain homotopy. We are using the local chart to identify the two instances of $V^m$ and $V^{m-1}$ here. The last statement in the lemma follows from the perturbation invariance from the previous lemma. Indeed, for any small change in $q,V^m$ and $v$ the perturbation invariance implies that the map induced on homology is locally constant in any local trivializations of the local systems.
\end{proof}

In the above lemma there are essentially two different isomorphisms for fixed $V^{m-1}\subset V^m$ - one for $v$ and one for $-v$. However, when dealing with a birth-death bifurcation, which of these is involved is uniquely determined by how the two critical points cancel. 

\begin{corollary}\label{cor:Fiberwise:2}
  Let $f :N \to \R$ be a function such that $f^{-1}[a,b]$ has precisely two critical points $q_0$ and $q_1$ in its interior. Assume also that these are non-degenerate and that there is a unique gradient trajectory between them so that they cancel in Morse homology. Assume $q_0$ is the one with the lower index and denote by $V^{m-1} \subset T_{q_0}N$ and $V^{m}\subset T_{q_1}N$ the negative eigenspaces of the Hessian of $f$ at the points.

  This data defines a chain homotopy equivalence
  \begin{align*}
    CF_*(L,q_1,V^m;C) \simeq CF_{*-1}(L,q_0,V^{m-1};C)
  \end{align*}  
  which is homotopic to parallel transport (continuation maps) composed with one of the two induced by the above lemma (determined by the cancellation).
\end{corollary}

\begin{proof}
  Define $K_t=tK$ and $L_t=tL+df$ as above but now using $f$. For very small $t$ we can argue as follows. Both of the chain complexes in the corollary are defined as different parts of the standard intersection Floer chain complex of $K_t$ and $L_t$. Considering only those intersection points with action in $[a,b]$ we get a chain complex which precisely contains these two parts. The differential restricted to this complex is again on upper triangular form (as in the proof above):
  \begin{align*}
    d =
    \begin{pmatrix}
      d^m & F \\
      0 & d^{m-1} 
    \end{pmatrix}
  \end{align*}
  Again $d^m$ and $d^{m-1}$ are the differentials in the fiber-wise Floer homology chain complexes $CF_*(K,L,q_0,V^{m-1};C)$ and $CF_*(K,L,q_1,V^{m-1};C)$. By the assumptions we can deform $f$ inside $f^{-1}([a,b])$ through a single birth-death singularity to a situation with no critical points. This also deforms $K_t$ and $L_t$, and we get induced continuation maps from Lemma~\ref{lem:Fiberwise:7}. 
\end{proof}

By picking orientations of the unstable manifolds (corresponds to orientations of $V^m$ and $V^{m-1}$) the sign of the differential in the usual Morse chain complex for $f$ in the above lemma is determined by the direction of $v$ given by the cancellation. Indeed, the sign is given by whether $V^{m-1}\oplus \R[v]=V^m$ is orientation preserving or not.

\begin{lemma}
  \label{lem:Fiberwise:3}
  The fiber-wise intersection Floer homology $HF_*(L,q,V^m;C)$ satisfies Morse shifting. I.e. for fixed $q$ we have
  \begin{align*}
    HF_*(L,q,V^m;C) \cong HF_{*-m}(L,q,0;C)
  \end{align*}
  canonically defined by fixing a choice of orientation on $V^m$. The isomorphisms for the two different orientations differ by a sign.
\end{lemma}

\begin{proof}
  Any choice of ordered orthonormal basis $(v_1,\dots,v_m)$ for $V^m$ defines by Lemma~\ref{lem:Fiberwise:8} a sequence of chain homotopy equivalences and thus a chain of isomorphisms
  \begin{align*}
    HF_*(L,q,V^m;C) \cong HF_{*-1}(L,q,V^{m-1};C) \cong \cdots \cong HF_{*-m}(L,q,0;C)
  \end{align*}
  which by the fact that these are locally maps of local systems is locally constant in the choice of such a basis and thus only dependent on the orientation class that the basis defines. Thus the only thing left to prove is that these two choices of isomorphisms differ by a sign.

  We may assume that the dimension of $N$ is at least 2. This means that when considering the two isomorphisms $HF_{*+m}(L,q,V^m;C) \cong HF_*(L,q,0;C)$ we can factor through $HF_{*+2}(L,q,V^2;C)$ (notice this works even when $m=0$ and $m=1$ since we can go up using inverses). Hence we can determine the difference of the two maps by simply considering the difference using any two compositions of the maps $F_v^m$ in Equation~\eqref{eq:13}. So, to see that the maps for the two different choices of orientations only differ by a sign (on the level of homology) we consider two orthogonal directions $v_1,v_2 \in V^m$ and denote the complement of $v_1$ by $V_1^{m-1}$ and $v_2$ by $V_2^{m-1}$ respectively. We also denote the common complement of the plane they span in $V^m$ by $V^{m-2}=V_1^{m-1}\cap V_2^{m-1}$. Now consider the diagram
  \begin{align*}
    \xymatrix{
    CF_*(L,q,V^m;C) \ar[r]^{F^m_{v_1}} \ar[d]^{F^m_{v_2}} & CF_*(L,q,V_1^{m-1};C) \ar[d]^{F^{m-1}_{v_2}}\\
    CF_*(L,q,V^{m-1}_2;C) \ar[r]^{F^{m-1}_{v_1}} & CF_*(L,q,V^{m-2};C)
    }
  \end{align*}
  We now run the same type of argument as in the proof of Lemma~\ref{lem:Fiberwise:8} above but with the family of functions given by
  \begin{align*}
    g&\strut_{s_1,s_2}(x) = \\
    &-x_0^2 - \cdots - x_{m-2}^2 + (x_{m-1}^2-s_1)x_{m-1} +(x_m^2-s_2)x_m+x_{m+1}^2 + \cdots + x_n^2
  \end{align*}
  using the two directions. Here we see by similar action arguments and with $s_1=s_2=-\epsilon$ that we get a family of acyclic complexes. For $s_1=s_2=\epsilon$ we get (for small $t$) that the differential can be written as
  \begin{align*}
    d =     
    \begin{pmatrix}
      d^m & F_{v_1}^m & F_{v_2}^m & H^{m}_{v_1,v_2} \\
      0 & d_1^{m-1} & 0 & F_{v_2}^{m-1} & \\
      0 & 0 & d_2^{m-1} & F_{v_1}^{m-1} & \\
      0 & 0 & 0 & d^{m-2}
    \end{pmatrix}
  \end{align*}
  Indeed, we have four critical points of $g$, but two of them have the same critical value. It is an easy application of Lemma~\ref{lem:Fiberwise:1} to see that for small $t$ the entries at position $(3,2)$ and $(2,3)$ are zero. Indeed, for small $t$ the action interval of each of the two bunches narrows around the same value, but we get a lower bound from Lemma~\ref{lem:Fiberwise:1} (used on $g_{\epsilon,\epsilon}$ and one of the critical points) on any disc with endpoints in both - hence no such disc exists for small $t$. Similarly we can identify the other entries with the maps $F_{v_i}^{m'}$ (up to homotopy) since no interaction between the two bunches at the same action level means that if we tip a little to $s_1>s_2$ the two equal critical values becomes slightly different, and then for small $t$ we get the four bunches at different action level. Then Corollary~\ref{cor:Fiberwise:2} implies that the map induced on the fiber-wise complexes from the bunch at the highest critical value to the bunch at the second highest critical value is $F_{v_1}^m$ (up to chain homotopy). Similarly for the other three maps.

  Now the fact that this differential squares to 0 gives that $H^m_{v_1,v_2}$ is a chain homotopy equivalence from $F^{m-1}_{v_2}\circ F^m_{v_1}$ to $-F^{m-1}_{v_1}\circ F^m_{v_2}$.
\end{proof}

\begin{lemma}
  \label{lem:Fiberwise:4}
  The fiber-wise intersection Floer homology satisfy Poincare duality
  \begin{align*}
     HF_*(L,q,V^m;C)^\dagger \cong HF_{n-*}(L,q,(V^m)^\perp;C^\dagger)
  \end{align*}
  where $C^\dagger$ denotes the fiber-wise dual local system over $L$ tensor the rank 1 local system of orientations on $L$.
\end{lemma}

\begin{proof}
  Firstly, if we replace $g$ by $2g$ in the definition of $L_t$ and apply the Hamiltonian flow of $-g\circ \pi : T^*N \to \R$ for time $1$ then $K_t$ and $L_t$ are flowed to the two Lagrangians
  \begin{align*}
    Q_t = tL - dg \qquad \textrm{and} \qquad P_t = tL + dg.
  \end{align*}
  This is more symmetric, and for small $t$ these can be used to define the Fiber-wise Floer homology. Indeed, it does not matter that we are using $2g$ nor does the local Floer homology change when we apply the Hamiltonian isotopy (the bunching of all critical points near the intersection, and the bound in Lemma~\ref{lem:Fiberwise:1} is valid throughout the isotopy for sufficient small $t$). The primitives used on these can be chosen as:
  \begin{align*}
    f^{P_t}(z) & = t f^L(t^{-1} (z-dg_{\pi(z)})) + g(\pi(z))   \\
    f^{Q_t}(z) & = t f^L(t^{-1} (z+dg_{\pi(z)})) - g(\pi(z)).
  \end{align*}
  If we now exchange $g$ for $-g$ we get the exact same two Lagrangians from this construction, but in the opposite order. Hence we change the sign of the action and simultaneously the directions of the pseudo-holomorphic discs counted in the differential.

  Since the Floer intersection chain complexes are given by a finite direct sum of fibers of the local system (which each may be infinite dimensional) the dual complex is a finite sum over the dual fibers of the local systems. Hence if we choose trivializations of $\absv{o_z}$ for all $z\in P_t\cap Q_t$ we can use the fact that the Lagrangians are the same to identify
  \begin{align*}
     CF_*(L,q,V^m;C)^\dagger \cong CF_{n-*}(L,q,(V^m)^\perp;C^\dagger)
  \end{align*}
  as vector spaces. The differential differs in signs by introducing the local system of orientations on $L$ since this ``inversion'' realizes Poincare duality of $L$ (see \cite{MR2441780} section 12 for details on these signs and the Poincare duality).
\end{proof}

For the proofs in Section~\ref{sec:consequences} the most important consequence of this section is the following ``0-dimensional'' Poincare duality for the fiber-wise Floer homology.

\begin{corollary} \label{cor:Fiberwise:1}
  The shift and the Poincare duality properties imply that
  \begin{align*}
    HF_*(L,q,0;C)^\dagger \cong HF_{-*}(L,q,0;C^\dagger).
  \end{align*}
  depending on a choice of orientation of $T_qN$.
\end{corollary}


\section{The spectral sequence}\label{sec:spectral-sequence}

In this section we construct the spectral sequence described in the introduction. However, as mentioned we will not do this for an arbitrary Morse function $g:N \to \R$. So, we start by describing the Morse function we are going to use in more detail.

By taking product with a large dimensional sphere we may assume that $N$ has dimension at least 6. Indeed, if the dimension is less than 6 then all the results follow from the same results for $L\times S^9 \subset T^*(N\times S^9)$. So, we may pick a Morse function $g : N \to \R$ and a pseudo-gradient $X : N \to TN$ such that:
\begin{itemize}
\item The pair is Morse-Smale,
\item the function $g$ is self-indexing (i.e. Morse index = critical value), and
\item if $x$ and $y$ are critical points of $g$ with adjacent Morse indices then there are either no pseudo-gradient trajectories connecting them or precisely 1.
\end{itemize}
Note that this last requirement can always be accomplished - by introducing a birth of two critical points along any unwanted gradient trajectory. This replaces a single gradient trajectory with 3, but also introduces two new critical points of which one can control the rigid trajectories down to lower dimensional strata (see e.g. \cite{MR0190942}).

Let $q_i$ denote the critical points of $g$. As in Section~\ref{sec:fiber-wise-self} let $L$ be an exact Lagrangians and define $K_t=tL$ and $L_t=tL+dg$. However, in this section we consider the global situation for this specific $g$ and do not focus on a specific critical point $q$. We therefore (for small $t$) introduce the filtration on the entire complex:
\begin{align*}
  F^p(CF_*(K_t,L_t;C))
\end{align*}
given by restricting to all the intersection point with action less than $p+\tfrac12$. We are now suppressing all small perturbations needed to properly define these. The continuation maps for perturbations of these will preserve the filtration as long as the action of an intersection point never crosses $p+\tfrac12$ for any $p$. Since the action values of the intersection points bunch around the critical values of $g$ (all integers) for small $t$ this is true for small $t$.

This filtration defines a spectral sequence converging to the intersection Floer homology of $L$ with $L$ with coefficients in $C$.

\begin{proposition}
  \label{prop:Spectral:1}
  Page 1 of this spectral sequence is isomorphic as a bi-graded chain complex to $CM_{*_1}(g;HF_{*_2}(L,q,0;C))$.
\end{proposition}

Here $CM_*(g;A)$ denotes the Morse homology complex of $g$ using the pseudo-gradient $X$ with coefficients in the graded local system $A$. Notice, that unlike the fiber bundle example above this may be non-trivial in negative $*_2$-gradings.

\begin{proof}
  Page one of such a spectral sequence has entry in bi-grading $(p,d)$ equal to the $(d+p)^\textrm{th}$ homology group of the quotient
  \begin{align*}
    (F^p(CF_*(K_t,L_t;C)) / F^{p-1}(CF_*(K_t,L_t;C)),\partial) = : C_{p,*}.
  \end{align*}
  For small $t$ the intersection points $K_t\cap L_t$ will cluster around the critical points $q_i$ and their action values will be close to the associated critical value $g(q_i)$ (see Section~\ref{sec:fiber-wise-self}). This critical value is the Morse index since $g$ is self indexing. The differential on each of the bunches around different critical points $q_i$ and $q_j$ with the same critical value cannot interact. Indeed, for small $t$ this would violate the energy bound from Lemma~\ref{lem:Fiberwise:1} on discs with one marked point sent to one bunch and the other to the other bunch. We thus get that the above quotient complex splits as a direct sum of the fiber-wise chain complexes from Section~\ref{sec:fiber-wise-self}:
  \begin{align*}
    C_{p,*} = \bigoplus_{
      \begin{subarray}{c}
        q_i \textrm{ critical} \\
        g(q_i)=p
      \end{subarray}
    } CF_*(L,q_i,V^m_i;C).
  \end{align*}
  Here $V^m_i$ is the negative eigenspace of the Hessian of $g$ at $q_i$. The homology of each of these are by Lemma~\ref{lem:Fiberwise:3} isomorphic and shifted by the Morse index (which by the self-indexing property equals $p$). By this and Lemma~\ref{lem:Fiberwise:2} we get
  \begin{align*}
    H_{d+p}(C_{p,*}) \cong \bigoplus_{
      \begin{subarray}{c}
        q_i \textrm{ critical} \\
        g(q_i)=p
      \end{subarray}
    } HF_{d}(L,q_i,0;C).
  \end{align*}
  For each summand this isomorphism depends on a choice of orientation of the unstable manifold at the critical point, but that is as it should be (since Morse homology works that way). Indeed, to argue what the differential (on page 1) is, we need to be careful with orientations (comparing with a CW complex structure on $N$ one needs to pick orientations of each cell before we can define the degree of attaching maps).

  The differential on page 1 of the spectral sequence is independent of $t$ for small $t$. Indeed, since there can be no interactions between the individual bunches of critical points (associated to the same Morse index) there can be no handle slides for small $t$.
  
  Fix $q_i$ and $q_j$ critical for $g$ with adjacent Morse indices, i.e. $p-1=g(q_j)<g(q_i)=p$. For very small $t$ we can pick a very small $\delta'$ and change $g$ by a small perturbation such that
  \begin{itemize}
  \item the critical value of $q_i$ becomes $p-\delta'$ and
  \item the critical value of $q_j$ becomes $p-1+\delta'$.
  \end{itemize}
  We can do this such that the change that this makes to $K_t$ and $L_t$ does not affect the identification above of page 1. Indeed for very small $t$ and $\delta'$ there are no possible interactions between any of the bunches approximately on the same action level (i.e. no disc can go from one to the other) - even while we push the action level of some of them up or down a little bit (the area bound $a>0$ in Lemma~\ref{lem:Fiberwise:1} can be assumed to be much larger than $\delta'$). It also does not affect the differential that we wish to identify. Indeed, any handle sliding is ruled out by the same argument.

  Now by making $t$ even smaller (which again does not change the above identification) we can make sure that the clustering around the critical point values is such that
  \begin{itemize}
  \item The intersection points in the bunch close to $q_j$ have action in the interval $p-1+[2\delta'/3,4\delta'/3]$,
  \item The intersection points in the bunch close to $q_i$ have action in the interval $p-[2\delta'/3,4\delta'/3]$, and
  \item The intersection points in bunches close to all other critical points have action in the intervals $\N + [-\delta'/3,\delta'/3] \subset \R$.
  \end{itemize}
  This means that there are no critical action values close to $p-1+\delta'/2$ and $p-\delta'/2$. The identification of the differential now follows by considering the chain complex defined by restricting to action between $p-1+\delta'/2$ and $p-\delta'/2$. Indeed, this is either:
  \begin{itemize}
  \item The birth-death situation we considered in Corollary \ref{cor:Fiberwise:2} (if there is a single gradient trajectory between the associated critical points).
  \item Or a situation where we can actually move the lower bunch up to the same height as the other and see that the differential on the fiber-wise homology has to be $0$ (by homotopy invariance). Indeed, if there are no gradient trajectories between the two critical points of $g$ we can by changing $g$ close to the unstable and stable manifolds move the critical points of $g$ in this way (see e.g. \cite{MR0190942}).
  \end{itemize}
  The first point uses the isomorphisms we saw in Lemma~\ref{lem:Fiberwise:3} and Corollary~\ref{cor:Fiberwise:2}, which has a sign depending on whether this cancellation is compatible with the chosen orientations on the unstable manifolds or not, which precisely is one way of defining the signs in $CM_*(g;A)$. So this is the Morse complex differential with the local coefficient system $HF_*(L,\bullet,0;C)$.
\end{proof}


\section{Local systems on the universal cover of $N$} \label{sec:local-syst-univ}

With the same assumptions as in Section~\ref{sec:fiber-wise-self} we will in this section define versions of the fiber-wise intersection Floer homology on the universal covering space of $N$ and prove compatibility with pull back and push forward maps. Then we will generalize Corollary~\ref{cor:Fiberwise:1} to dualizing the local systems on the universal covers.

Most of the results in this section are easy consequences of the following corollary to Lemma~\ref{lem:Fiberwise:1}. However, the introduced language and notation will be convenient for the general proof of Theorem~\ref{thm:1}.

Let $\pi_N:N'\to N$ be the universal covering space of $N$. To this we have an associated universal covering $T^*N' \to T^*N$.

\begin{corollary} \label{lem:Local:2}
  Assume all the conditions of Lemma~\ref{lem:Fiberwise:1} - except assume that $u$ has both points $u(\pm 1)$ mapping to $T^*B_R(q)$ instead of precisely one of them. Additionally assume that $u$ has energy less than $a$. Then $u$ is homotopic in $T^*N$ relative to $\{\pm 1\} \subset D^2$ to a map in $T^*B_R(q)$.
\end{corollary}

\begin{proof}
  Since the disc relative the points is homotopy equivalent to the interval relative its endpoints this is a question of what $u$ represents in
  \begin{align*}
    \pi_1(T^*N,T^*B_R(q)) \cong \pi_1(T^*N,q).    
  \end{align*}
  However, assuming it represents something non-trivial the disc will lift to have two endpoints in $T^*N'$ which are in two different components of the non-connected pre-image of the contractible sub-space $T^*B_R(q)$. Hence as in the proof of Lemma~\ref{lem:Fiberwise:1} this intersects the pre-image of $W$ (from that proof) in $T^*N'$. This implies that $u$ in fact intersects $W$ non-trivially, which gives a contradiction (if $a$ is chosen as in that proof).
\end{proof}

Define the covering space $\pi_L: L'\to L$ by the pull back diagram
\begin{align} \label{eq:6}
  \xymatrix{
    L' \ar[r]^{j'} \ar[d]^{\pi_L}  &  T^*N' \ar[d] \\
    L \ar[r]^{j} &  T^*N
  }
\end{align}
Note that a priori $L'$ can have more components than $L$. Indeed, we are not lifting the map - we are taking the pull back. In the following we will refer to $L \subset T^*N$ as ``downstairs'' and $L' \subset T^*N'$ as ``upstairs''. Any (graded) local system on $N$ or $L$ can be pulled back to a (graded) local system on $N'$ or $L'$ by $\pi_N$ or $\pi_L$ respectively. However, recall that we only consider local systems on $L$ which have support in degree $0$.

Let $C' \to L'$ be a local system of $\F$-vector spaces. The push forward $\pi_{L*} C'$ is the local system on $L$ defined by
\begin{align*}
  (\pi_{L*}C')_z = \bigoplus_{z'\in \pi^{-1}(z)} C_{z'}'.
\end{align*}
We define the intersection Floer homology $HF_*(L',L';C')$ as the Floer homology $HF_*(L,L;\pi_{L*} C')$, which means that we have an associated fiber-wise Floer homology
\begin{align*}
  HF_*(L,\bullet,0;\pi_{L*} C')
\end{align*}
as in Section~\ref{sec:fiber-wise-self}. We may similarly use $\pi_N$ to push forward graded local systems on $N'$ to $N$. We now also define a version of the fiber-wise Floer homology on the covering space which we will see is compatible with both push forward and pull back maps.

Let $(q,V^m,g,t,H,J')$ be all the data needed to define an instance of the fiber-wise complex associated to the Floer homology with coefficients $\pi_{L*}C'$ at some point $q\in N$. The reader may heuristically consider this as a $\pi_1$-equivariant perturbation on the covering space. Now let $q' \in \pi^{-1}(q)$ be a choice of lift of $q$ then as a graded vector spaces we define
\begin{align} \label{eq:7}
  CF_*(L',q',V^m;C') = \smashoperator{\bigoplus_{z'\in L'\cap L' \cap T^*B_{R/2}(q')} } \absv{o_{\pi_L(z')}} \otimes C'_{z'}
\end{align}
(after the perturbation). By definition we have an isomorphism as graded vector spaces
\begin{align} \label{eq:9}
  CF_*(L,q,V^m;\pi_{L*}C') \cong  
   \smashoperator{\bigoplus_{q'\in \pi^{-1}(q)}} CF_*(L',q',V^m;C') = \pi_{N*} CF_*(L',q',V^m;C')
\end{align}
where both sides are defined using the same perturbation data. However, the content of Corollary~\ref{lem:Local:2} is that the differential actually respects this splitting. So the graded vector space in Equation~\eqref{eq:7} is naturally a chain complex, and we define the fiber-wise intersection Floer homology $HF_*(L',q',V^m;C')$ as its homology. It follows precisely as before that this defines a graded local system on the choices of $q'\in N'$ together with an $m$ dimensional subspace $V^m \subset T_{q'}N$.

By this definition we now have two natural isomorphisms:
\begin{align} \label{eq:11}
  HF_*(L,\bullet,0;\pi_{L*} C') \cong \pi_{N*} HF_*(L',\bullet,0;C')
\end{align}
for any local system $C'$ on $L'$ and
\begin{align} \label{eq:12}
  HF(L',\bullet,0;\pi^*_L C) \cong \pi^*_N HF(L,\bullet,0;C).
\end{align}
for any local system $C$ on $L$.

We have the following generalization of Floer's result and the spectral sequence in Proposition~\ref{prop:Spectral:1}.

\begin{lemma} \label{lem:Local:1}
  We have
  \begin{align*}
    HF_*(L',L';C') \cong H_*(L';C')
  \end{align*}
  and the associated spectral sequence in Proposition~\ref{prop:Spectral:1} can on page 2 canonically be identified with
  \begin{align*}
    H_*(N';\F) \otimes HF_*(L',\bullet,0;C').
  \end{align*}
\end{lemma}

\begin{proof}
  This is an easy consequence of Equation~\eqref{eq:11}, Equation~\eqref{eq:12}, the fact that $N'$ is simply connected, and the fact that
  \begin{align*}
    H_*(N';A) \cong H_*(N;\pi_{N*}A) \cong HM_*(g;\pi_{N*}A)
  \end{align*}
  for any graded local system $A$ on $N'$.
\end{proof}

For any local system $C'$ on $L'$ we define (similar to the definition in Section~\ref{sec:fiber-wise-self}) its dual $C'^\dagger$ over $L'$ to be the fiber-wise dual vector space tensored with the local system defined by orientations on $L'$. Notice that even if $L$ is orientable $\pi_{L*}(C'^\dagger)$ is not generally isomorphic to $(\pi_{L*}C')^\dagger$ if $\pi_1(N)$ is not finite.

The above observations now makes it possible to generalize the Poincare duality from Corollary~\ref{cor:Fiberwise:1} to this dualization.

\begin{corollary}\label{cor:Local:1}
  For fixed $q'\in N'$ we have an isomorphism
  \begin{align*}
    HF_*(L',q',0;C'^\dagger) \cong HF_{-*}(L',q',0;C')^\dagger.
  \end{align*}
\end{corollary}

Since these are trivial local systems we do not really need to fix $q'$. However, the following proof is easier to mentally parse downstairs when $q=\pi_N(q')$ is considered a fixed point.

\begin{proof}
  By considering the above definition (using Corollary~\ref{lem:Local:2}) of the differential of $CF_*(L',q',0;C')$ as counting discs downstairs in $T^*N$ all the proofs (considering the point $q=\pi_N(q')$ fixed) in Section~\ref{sec:inters-floer-homol} generalizes to this case.
\end{proof}


\section{Proof of Theorem 1}\label{sec:consequences}

This section contains a proof of Theorem~\ref{thm:1}. So assume $L\subset T^*N$ is an exact Lagrangian with vanishing Maslov class. In this section $g$ is a function as in Section~\ref{sec:spectral-sequence} such that Proposition~\ref{prop:Spectral:1} holds for this $g$. 

As a warm up we start by giving a proof of homotopy equivalence in the case $\pi_1(N)=1$ and connected $L$. This is similar to the argument given in \cite{MR2596633} - except that instead of using the notion of the span of the homology we use the fiber-wise Poincare duality in Corollary~\ref{cor:Fiberwise:1}. This is also one reason we are able to get stronger results.

As in \cite{MR2596633} we start by using coefficients $\F_2=\Z/2$. This means that without assumptions the intersection Floer homology is defined. The trivial fundamental group implies that the local system $HF_*(L,\bullet,0;\F)$ is trivializable over $N$. This means that Proposition~\ref{prop:Spectral:1} implies that the spectral sequence (with trivial local system $C=\F_2$ on $L$) on page two is isomorphic to
\begin{align*}
  H_{*_1}(N;\F_2)\otimes HF_{*_2}(L,q_0,0;\F_2).
\end{align*}
Since the higher differentials cannot kill the degree $(0,*_2)$ with $*_2$ the lowest degree where $HF_{*_2}(L,q_0,\F_2)$ is supported this has to survive to page infinity. This leads to a contradiction if this $*_2$ degree is negative. Indeed, the homology of $L$ is supported in positive degrees. So
\begin{align*}
  HF_{*_2}(L,q_0;\F_2) = 0 \qquad *_2<0.
\end{align*}
Now the Poincare duality in Corollary~\ref{cor:Fiberwise:1} implies that the support is purely in degree $0$ ($L$ is oriented with respect to $\F_2$). It follows that the spectral sequence collapses on page 2 and that
\begin{align*}
  H_*(L;\F_2) \cong H_*(N;\F_2^{\oplus k}), 
\end{align*}
where $k$ necessarily equals the rank of $HF_0(L,q_0;\F_2)$, which by assumption is 1. This implies that there is an abstract graded isomorphism between the $\F_2$-homologies of $L$ and $N$, and that the Euler characteristic of $CF_*(L,q,0;\F_2)$ is 1. The latter implies that $L\to N$ has degree 1 (see the general argument below for details on this part), which implies that the induced map $H_*(L;\F_2) \to H_*(N;\F_2)$ is surjective. Combining this with the knowledge of an abstract isomorphism between the two we get that $L\to N$ is an $\F_2$ homology equivalence.

Now as noted in \cite{MR2596633} this implies that the map $L\to N$ is relatively orientable and relative spin, and that $L\subset T^*N$ has a relative pin structure so that we can define the homologies with $\F$ coefficients for any field. Now the exact same argument proves homology equivalence over any field, which implies homology equivalence over $\Z$.

We wish to also prove that $\pi_1(L)$ is trivial. So, for contradiction assume that this is not the case. Then we have a nontrivial cyclic sub-group $G\subset \pi_1(L)$ This has an associated covering space $\widetilde{L} \to L$ with $\pi_1(\widetilde{L})=G$ and hence $H_1(\widetilde{L},\F)\neq 0$ for some field $\F$. Let $C$ be the local system of $\F$ vector spaces which has $H_*(L;C) \cong H_*(\widetilde{L},\F)$. Consider page 2 of the spectral sequence from Proposition~\ref{prop:Spectral:1} using these coefficients:
\begin{align*}
  H_*(N,\F) \otimes HF_*(L,\bullet,0;C).
\end{align*}
This converges in the abutment to $H_*(\widetilde{L};\F)$, and as above we conclude that $HF_*(L,\bullet,0;C)=0$ for $*<0$. Since $H_*(\widetilde{L},\F)$ has non-trivial $H_1$ and $N$ is simply connected we conclude that $HF_1(L,q,0;C) \neq 0$. Now the Poincare duality in Corollary~\ref{cor:Fiberwise:1} shows that 
\begin{align*}
  HF_{-1}(L,q,0;C^\dagger)\cong HF_{1}(L,q,0;C)^\dagger\neq 0.
\end{align*}
However, again as above (using the spectral sequence with local system $C^\dagger$ on $L$) this is contradictory to the fact that $H_*(L;C^\dagger)$ has non-negative support and that 
\begin{align*}
   H_*(N;\F) \otimes HF_*(L,q,0;C^\dagger)
\end{align*}
converges to it in the abutment.

For the general proof of Theorem~\ref{thm:1} we divide the argument into a few lemmas. 

We are no longer assuming that $L$ is connected. Let $L'\to T^*N'$ be as in Equation~\eqref{eq:6}. The Lagrangian $L'$ can have more components than $L$. Firstly we consider the trivial local system $\F$ on $L'$. The push-forward of this to $L$ is the the same as the pull-back of the local system on $N$ which represents the universal covering space $N' \to N$. So, we denote this by $C^N$. Corollary~\ref{cor:Grading:1} shows that
\begin{align} \label{eq:5}
  HF_*(L,L;C^N) \cong H_*(L;C^N) \cong H_*(L';\F)
\end{align}
when defined (relative pin structure required when $\chr \F \neq 2$).

\begin{lemma}\label{lem:Results:1}
  When defined the Fiber-wise Floer homology is concentrated in degree 0 and 
  \begin{align*}
    HF_0(L,q,0;\F) \cong \F^k
  \end{align*}
  where $k$ is the rank of $H_0(L';\F)$. In particular this implies that this rank is finite.
\end{lemma}

\begin{proof}
  By Equation~\eqref{eq:12} we have
  \begin{align*}
    HF_*(L',\bullet,0;\F) \cong \pi_{N}^*HF_*(L,\bullet,0;\F),
  \end{align*}
  so we may prove the statement in the lemma for this local system on $N'$.

  By Lemma~\ref{lem:Local:1} we have that
  \begin{align*}
    H_*(N';\F) \otimes HF_*(L',\bullet,0;\F)
  \end{align*}
  is page two of a spectral sequence converging to $H_*(L';\F)$. This implies that
  \begin{align*}
    HF_0(L',\bullet,0;\F) \cong H_0(L';\F),
  \end{align*}
  and that this Fiber-wise homology is supported in non-negative degree. Similarly, the spectral sequence for the local system $\F^\dagger$ (dual over $L'$) converging to $H_*(L';\F^\dagger)$ shows that $HF_*(L',\bullet,0;\F^\dagger)$ is trivial in negative degrees. By Corollary~\ref{cor:Local:1} this implies that $HF_*(L',\bullet,0;\F)$ is trivial in positive degrees.
\end{proof}

\begin{lemma} \label{lem:Results:3}
  Both $L$ and $L'$ are connected. In particular
  \begin{align*}
    HF_*(L;\bullet,0;\F) \cong \F
  \end{align*}
  is defined and is the trivial local system for any $\F$. Furthermore, $j:L \to T^*N$ induces a homology equivalence and a surjection on $\pi_1$.
\end{lemma}

\begin{proof}
  Firstly assume that $\F=\F_2$. The vanishing of the Maslov class implies that $L'$ is orientable (since $N'$ is). For some orientations on $L'$ and $N'$ let $p$ denote the degree of the map $L' \to N'$ at a generic $q'$ defined by the sum of $\pm 1$ associated with the orientations of the linear isomorphisms:
  \begin{align*}
    D_{x}j' : T_xL' \to T_{q'}N'
  \end{align*}
  for all $x\in \pi^{-1}(q') \subset L'$. We will call a layer of $L'$ positive if it contributes positively to this and negative otherwise. Note that this is independent of $q'$ since the map $L' \to N'$ is proper and $N'$ is connected. For any orientation on $N'$ we may pick the orientation on each component of $L'$ such that each component contributes non-negatively to the degree. Now, let $q'\in N'$ be a lift of a global minimum $q\in N$ of $g$. We may assume that the cotangent fiber $T^*_qN$ is transverse to $L$.

  We can compute the Euler characteristic of $HF_*(L,q,0;\F)$ as $p^2$. Indeed, for small $t$ we see that the Lagrangians $K_t=tL+dg$ and $L_t=tL$ will be transverse to each other, and the parity of the Maslov index of an intersection can be computed using the orientation sign of the two layers of the lift associated with the intersection (see figure~\ref{fig:1}).
  \begin{figure}[ht]
    \centering
    \begin{tikzpicture}
      \draw[->] (-2,0) -- (2,0) node [right] {$N'$};
      \draw[->] (0,-1.4,0) -- (0,1.7) node [below left] {$T^*_{q'}N$};
      \draw (0,0) node [below right] {$q'$};
      \draw plot [smooth] coordinates {(-2,1.1) (0.8,0.7) (-0.8,-0.5) (2,-0.9)};
      \fill (0,0.9) circle (1.4pt) node [above right] {$+$};
      \fill (0,0.1) circle (1.4pt) node [left] {$-$};
      \fill (0,-0.7) circle (1.4pt) node [below left] {$+$};
      \draw (0.8,0.7) node [above right] {$L$};
    \end{tikzpicture}
    \begin{tikzpicture}
      \draw[->] (-2,0) -- (2,0) node [right] {$N'$};
      \draw[->] (0,-1.4,0) -- (0,1.7) node [below left] {$T^*_{q'}N$};
      \draw (0,0) node [below right] {$q'$};
      \draw plot [smooth] coordinates {(-1.5,1.59) (0.8,-0.73) (-0.8,0.75) (1.5,-1.57)};
      \draw plot [smooth] coordinates {(-2,0.11) (0.8,0.07) (-0.8,-0.05) (2,-0.09)};
      \draw (-0.8,0.7) node [above right] {$L_t$};
      \draw (-0.8,0) node [below] {$K_t$};
      \draw[dashed] (0,0) circle (0.4);
      \draw[dashed] (0.35,0.1) -- (2.5,0.5);
      \draw[dashed] (3.5,0.5) circle (1.0);
      \draw[dotted] (2.6,0.7) -- (4.4,0.7);
      \draw[dotted] (2.6,0.5) -- (4.4,0.5);
      \draw[dotted] (2.6,0.3) -- (4.4,0.3);
      \draw[dotted] (2.85,1.15) -- (4.15,-0.15);
      \draw[dotted] (3.0,1.3) -- (4.3,0.00);
      \draw[dotted] (2.7,1) -- (4,-0.3);
      \draw (3.0,0.7) node {$+$};
      \draw (3.3,0.7) node {$-$};
      \draw (3.6,0.7) node {$+$};
      \draw (3.2,0.5) node {$-$};
      \draw (3.5,0.5) node {$+$};
      \draw (3.8,0.5) node {$-$};
      \draw (3.4,0.3) node {$+$};
      \draw (3.7,0.3) node {$-$};
      \draw (4.0,0.3) node {$+$};
    \end{tikzpicture}
    \caption{Intersection signs giving Maslov parity (also indicated by a sign)}.
    \label{fig:1}
  \end{figure}
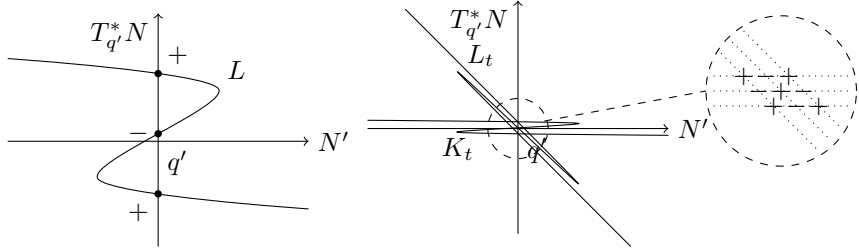
 If $L'$ has $p+k$ positive layers and $k$ negative layers at $q'$ we therefore get $(p+k)^2+k^2$ even parity Maslov indices of intersection points and $2(p+k)k$ odd parity Maslov indices of intersections points. Hence the Euler characteristic of the complex is $p^2$. This implies together with Lemma~\ref{lem:Results:1} that in fact $H_*(L';\F) \cong HF_0(L,q,0;\F) \cong \F^{p^2}$.

  Now assume $L$ can be divided into two components $L_1 \cup L_2$ (each not necessarily connected). Then we can do the same as above, but for each $L_i$ and its covering spaces $L_i'$. Call the degree of each lifts $p_i$ (with the same choices of orientations as the previous paragraph) then $p=p_1+p_2$. The same argument for $L_1$ and $L_2$ as distinct Lagrangians shows that $p_1^2=\rank H_0(L_1')\neq 0$ and $p_2^2=\rank H_0(L_2')\neq 0$. This gives that
  \begin{align*}
    \rank (H_0(L')) = &p^2 = (p_1+p_2)^2 = p_1^2 + p_2^2 + 2p_1p_2 = \\
     = & \rank(H_0(L_1')) + \rank(H_0(L_2')) + 2p_1p_2
  \end{align*}
  which is a contradiction since 
  \begin{align*}
    \rank(H_0(L'))=\rank(H_0(L_1')) + \rank(H_0(L_2'))
  \end{align*}
  So, $L$ is connected.

  Now assume that $p^2> 1$. Since $L$ is connected, and $L'$ is not, the map $\pi_1(L) \to \pi_1(N)$ is not surjective - in fact $\absv{\pi_1(N)/\im(\pi_1(L))} = \rank(H_0(L'))=p^2$. This means that there is a covering space of $N$ with $p^2$ layers (associated to the image sub-group) where the lift of $L$ has $p^2$ components, but such a lift of $N$ is finite and hence compact, and this contradicts the fact that we just proved that exact Lagrangians in such are connected. So, $p=1$ and even $L'$ is connected. We also conclude that the map $L \to N$ has degree 1. Note, that degree is defined by passing to oriented covers in the case where $L$ and $N$ are non-orientable.

  This means that the local systems $HF_*(L,q,0;\F)$ is free and of rank 1 with support in degree 0. Now assume for contradiction that it is not trivial. Then we get by the spectral sequence in Proposition~\ref{prop:Spectral:1} using $C=\F$ that $H_0(L;\F)\cong 0$, which is a contradiction. We conclude that $H_*(N,\F) \cong H_*(L,\F)$, but since we have not proven naturality with respect to $j$ we can only claim this as an abstract isomorphism. However, we \emph{have} proven that the map $j$ has degree 1, and a degree 1 map of closed manifolds is surjective with field coefficients, and so this abstract isomorphism shows (since the dimensions agree) that it is also injective.

  Now as before all this implies existence of relative pin structure on $L$, and we can therefore run the parts of the argument needed using a general field $\F$ to obtain homology equivalence.
\end{proof}

\begin{lemma}\label{lem:Results:2}
  The map on fundamental groups induced by $j$ is injective.
\end{lemma}

\begin{proof}
  Assume for contradiction that it has a kernel $\{1\} \neq \pi_1(L') \subset \pi_1(L)$, which as indicated by this notation is the fundamental group of the covering $L'$. Let $G\subset \pi_1(L')$ be a non-trivial cyclic sub-group of prime order (or order $\infty$) and $\widetilde{L} \to L'$ its corresponding covering space. Now let $C_G'$ denote the local coefficients over the field $\F_{\absv{G}}$ (with the convention $\F_\infty=\Q$) on $L'$ corresponding to this covering, and define $C_G=\pi_{L*}C_G'$. This is the local system on $L$ corresponding to the covering space $\widetilde{L}\to L'\to L$.
  
  Now this and Equation~\eqref{eq:11} shows that 
  \begin{align*}
    H_*(L;C_G) \cong H_*(L';C_G') \cong HF_*(L,L;C_G) \cong \pi_{N*} HF_*(L',L';C_G')
  \end{align*}
  Again by Lemma~\ref{lem:Local:1} page two of the spectral sequence associated to $C_G$ is
  \begin{align*}
    H_*(N')\otimes HF_*(L',\bullet,0;C_G')
  \end{align*}
  and converges to $H_*(L';C_g')$. Again we conclude that $HF_*(L',\bullet,0;C_G')$ has non-negative support, but also that it is non-trivial in degree 1. However, using Corollary~\ref{cor:Local:1} we see that this contradicts that $H_*(L';C_G'^\dagger)$ is supported in non-negative degree since the spectral sequence with the dual coefficients $HF_*(L',\bullet,0;C_g'^\dagger)$ on $N'$ converges to this.
\end{proof}


\bibliographystyle{plain}
\bibliography{Mybib}

\end{document}